\documentclass[a4paper]{amsart}
\usepackage{latexsym}

\usepackage{xcolor}
\usepackage[normalem]{ulem}

\usepackage{mathtools}

\usepackage{geometry}
\newtheorem{thm}{Theorem}[section] 
\newtheorem{lem}[thm]{Lemma}     
\newtheorem{cor}[thm]{Corollary}
\newtheorem{prop}[thm]{Proposition}
\newtheorem{conj}[thm]{Conjecture}

\theoremstyle{definition}

\numberwithin{equation}{section}

\def\Z{\mathbb{Z}}
\def\F{\mathbb{F}}
\def\Sym{\mathrm{Sym}}
\def\wt{\widetilde}

\begin{document}

\title[Some new results about a conjecture by Brian Alspach]{Some new results about \\ a conjecture by Brian Alspach}

\author{S. Costa}
\address{DICATAM - Sez. Matematica, Universit\`a degli Studi di Brescia, Via
Branze 43, 25123 Brescia, Italy}
\email{simone.costa@unibs.it}

\author{M.A. Pellegrini}
\address{Dipartimento di Matematica e Fisica, Universit\`a Cattolica del Sacro Cuore, Via Musei 41, 25121 Brescia, Italy}
\email{marcoantonio.pellegrini@unicatt.it}

\begin{abstract}
In this paper we consider the following conjecture, proposed by Brian Alspach, concerning partial sums in finite 
cyclic groups: given a subset $A$ of $\Z_n\setminus \{0\}$ of size $k$ such that $\sum_{z\in A} z\not= 0$, it is 
possible to 
find an ordering $(a_1,\ldots,a_k)$ of the elements of $A$ such that the partial sums $s_i=\sum_{j=1}^i a_j$,
$i=1,\ldots,k$, are nonzero 
and pairwise distinct.
This conjecture is known to be true for subsets of size $k\leq 11$ in cyclic groups of prime order.
Here, we extend such result to any torsion-free abelian group and, as a consequence, we provide an asymptotic result in 
$\Z_n$. 

We also consider a related conjecture, originally proposed by Ronald Graham:
given a subset  $A$ of $\mathbb{Z}_p\setminus\{0\}$, where $p$ is a prime, there exists an ordering of the elements 
of $A$ such that the partial sums are all distinct.
Working with the methods developed by Hicks, Ollis and Schmitt, based on the  Alon's combinatorial Nullstellensatz,
we prove the validity of such conjecture for subsets $A$ of size $12$.
\end{abstract}


\maketitle

\section{Introduction}

In this paper we give some new results about a conjecture, due to Brian Alspach, concerning finite cyclic groups. 
First of all, we introduce some notations.
Given an abelian group $(G,+)$, a finite subset $A=\{x_1,x_2,\dots, x_k\}$ of $G\setminus\{0_G\}$ 
and an ordering $\omega=(x_{j_1},x_{j_2},\ldots,x_{j_k})$ of its elements, we denote by $s_i=s_i(\omega)$ 
the partial sum $x_{j_1}+x_{j_2}+\dots+x_{j_i}$.
Clearly, the ordering $\omega$  induces the permutation $\sigma_\omega=(j_1,j_2,\ldots,j_k)\in \Sym(k)$.
Alspach's conjecture was originally proposed only for finite cyclic groups, see \cite{ADMS,BH}. 
However, it can be extended to any abelian group, \cite{CMPP, O1}.

\begin{conj}[Alspach]\label{Conj:als}
Given an abelian group $(G,+)$ and a subset $A$ of $G\setminus \{0_G\}$ of size $k$ such that $\sum_{z\in A} z\not= 
0_G$, 
it is possible to find an ordering $\omega$ of the elements of $A$ such that $s_i(\omega)\not=0_G$ and 
$s_i(\omega)\neq s_j(\omega)$ for all $1 \leq i < j \leq k$. 
\end{conj}

The validity of such conjecture has been proved in each of the following cases:
\begin{itemize}
\item[(1)] $k\leq 9$ or $k=|G|-1$, \cite{6,ADMS,BH,Gordon};
\item[(2)] $k=10$ or $|G|-3$ with $G$ cyclic of prime order, \cite{JOS};
\item[(3)] $k=11$  with $G$ cyclic of prime order, \cite{Sarah};
\item[(4)] $|G| \leq 21$, \cite{BH,CMPP};
\item[(5)] $G$ is cyclic and either $k=|G|-2$  or $|G|\leq 25$,  \cite{ADMS,BH}.
\end{itemize}
Clearly, when $k=|G|-3,|G|-2,|G|-1$, $G$ is assumed to be finite.
Alspach's conjecture is worth to be studied also in connection with sequenceability and strong sequenceability of 
groups, see \cite{Al,6,O}, and simplicity of Heffter arrays, see \cite{A, ADDY, CMPPHeffter, DW}.

In Section \ref{torsion} we explain how the validity of Conjecture \ref{Conj:als} for sets of size $k$ in cyclic 
groups
$\Z_p$, for infinitely many primes $p$, implies the validity for sets of size $k$ in any torsion-free abelian group. 
As a consequence, in Section \ref{asy}, we provide an asymptotic result for sets of size $k\leq 11$ 
in finite cyclic groups: this has been achieved without any direct or recursive construction 
(that we believe can hardly be obtained) but only with some theoretical non-constructive  arguments. 

Another conjecture, very close to the Alspach's one, was originally proposed by R.L. Graham in \cite{Gr} for cyclic 
groups of prime order, and by  D.S. Archdeacon, J.H. Dinitz, A. Mattern and D.R. Stinson for any finite cyclic group, see 
\cite{ADMS}.

\begin{conj}[G-ADMS]\label{Conj:ADMS}
Let $A\subseteq \mathbb{Z}_n\setminus\{0\}$. 
Then there exists an ordering of the elements of $A$ such that the partial sums are all
distinct.
\end{conj}

In \cite{ADMS} the authors proved that Conjecture \ref{Conj:als} in a group $\Z_n$ for sets of size at most $k$ 
implies Conjecture \ref{Conj:ADMS} in the same group $\Z_n$ for sets of size at most $k$.
As remarked in \cite{CMPP}, G-ADMS conjecture can be extended to any finite subset of an abelian group.
This implies that our results on Alspach's conjecture can be applied also to G-ADMS conjecture.
In Section \ref{altre}, we prove the validity of Conjecture \ref{Conj:ADMS} for subsets of size $12$
of cyclic groups of prime order. This result is achieved using 
Alon's combinatorial Nullstellensatz and the techniques developed in \cite{JOS}.
As a consequence we obtain a similar extension to torsion-free abelian groups 
and a similar asymptotic result.

\section*{Acknowledgements}

The computations of Section \ref{altre} have been performed with Magma \cite{mg} 
on a computer with 32 GB of RAM, kindly provided by Francesco Strazzanti:
the authors want to sincerely thank Francesco for his help.

The first author was partially supported by the National Group for Algebraic and Geometric Structures, and their 
Applications (GNSAGA--INdAM).

\section{Alspach's conjecture for torsion-free abelian groups}\label{torsion}

In this paper, we will say that a finite subset $A$ of an abelian group $(G,+)$ 
is nice if $0_G \not \in A$ and  $\sum_{z\in A} z\not= 0_G$. Also, with an abuse of notation, 
we will say that Alspach's conjecture is true in $G$ for any subset of 
size $k$ if it is true for any nice subset of size $k$.

Given a nice subset $A$ of an abelian group $G$, 
by $\Delta(A)$ we mean the set $\{x_1-x_2 :   x_1,x_2\in A, x_1\neq x_2\}$.
This allows us to define the set 
$$\Upsilon(A)=A\cup\Delta(A)\cup\left\{\sum_{z\in A} z\right\}.$$
Given two integers $a,b$ with $a\leq b$, the subset $\{a,a+1,\ldots,b\}\subset \Z$ will be denoted by $[a,b]$.
Our aim is to extend the known results on Alspach's conjecture in cyclic groups of prime order  to any torsion-free 
abelian group.

\begin{lem}\label{cambiogruppo}
Let $G_1$ and $G_2$ be abelian groups such that Alspach's conjecture holds in $G_2$ for any subset of size $k$. 
Given a nice subset $A$ of $G_1$ of size $k$, suppose there exists an homomorphism 
$\varphi: G_1\to G_2$ such that $\ker(\varphi)\cap \Upsilon(A)=\emptyset$. 
Then, Alspach's conjecture is true for the subset $A$.
\end{lem}

\begin{proof}
Since no element of $\Upsilon(A)$ belongs to the kernel of $\varphi$, $\varphi(A)$ is a nice subset 
of $G_2$ of size $k$. Therefore, there exists an ordering 
$\omega_2=(x_1,x_2,\ldots,x_k)$ of the elements of $\varphi(A)$ such that $s_i(\omega_2)\not=0_{G_2}$ and 
$s_i(\omega_2)\neq s_j(\omega_2)$ for all $1 \leq i < j \leq k$. 
However, considering the ordering $\omega_1=(z_1,z_2,\ldots,z_k)$ of the elements of $A$, where $\varphi(z_i)=x_i$ for all 
$i=1,\ldots,k$, we obtain that the partial sums $s_i(\omega_1)$ in $G_1$ are still pairwise distinct and 
nonzero.  
\end{proof}

\begin{prop}\label{prop:Z}
Let $k$ be a positive integer and suppose that, for infinitely many primes $p$,  Alspach's conjecture holds in $\Z_p$ for 
any subset of size $k$. Then Alspach's conjecture holds in $\Z$ for any subset of size $k$.
\end{prop}

\begin{proof}
Consider a nice subset $A$ of $\Z$ of size $k$. Let $p>\max\limits_{z\in \Upsilon(A)} |z|$  be a prime such that 
Alspach's conjecture holds in $\Z_p$ for subsets of size $k$. 
Then, $\Upsilon(A)$ and  the kernel of the canonical projection $\pi_p: \Z\to  \Z_p$ are disjoint sets.
Hence, the statement follows from Lemma \ref{cambiogruppo}.
\end{proof}

We now consider the free abelian group $\Z^n$ of rank $n$.

\begin{prop}\label{prop:Zn}
Suppose that Alspach's conjecture holds in $\Z$ for any subset of size $k$. 
Then it holds in $\Z^n$ for any $n\geq 2$ and any subset of size $k$.
\end{prop}

\begin{proof}
Fix a nice subset $A=\{a^1,a^2,\dots,a^k\}$ of $\Z^{n}$ of size $k$, and set $B=\Upsilon(A)$. 
Given an integer 
$$M>\max\limits_{(z_1,\dots,z_{n}) \in B}\max\limits_{j\in [1,n]} n|z_j|,$$
we define the 
homomorphism $\varphi: \Z^{n} \to \Z$, as follows: 
$$\varphi(x_1,\dots,x_{n})=\sum\limits_{i=1}^{n} x_i M^{i-1}.$$
Because of the choice of $M$, the subset $B$ and the 
kernel of $\varphi$ are disjoint. Namely, suppose that there exists $b=(y_1,y_2,\ldots,y_n) \in B$ such 
that $\varphi(b)=0$.
We can assume that $y_{s_1}, \ldots,y_{s_c}$ are all nonnegative integers and that 
$y_{t_1}, \ldots,y_{t_d}$ are all negative integers.
Then, we can write 
$$\sum_{j=1}^c y_{s_j} M^{s_j-1} =\sum_{j=1}^d (-y_{t_j}) M^{t_j-1}.$$
We can look at the two sides of this equality as two expansions in base $M$ of the same nonnegative integer, 
since the coefficients $y_{s_1},\ldots,y_{s_c}, y_{t_1},\ldots,y_{t_d}$ all belong to the set $[0,M-1]$.
The uniqueness of such expansion implies that all these coefficients are zero, i.e., that $b=0$.
It follows from Lemma \ref{cambiogruppo} that Alspach's conjecture holds for the subset $A$.
\end{proof}

From the previous proposition we deduce this result.

\begin{thm}\label{torsionfree}
Let $k$ be a positive integer and suppose that, for infinitely many primes $p$,  Alspach's conjecture holds in $\Z_p$ for 
any subset of size $k$. Then Alspach's conjecture  holds for any  subset of size $k$ in any torsion-free abelian group 
$G$. 
\end{thm}

\begin{proof}
Let $A$ be a nice subset of $G$ of size $k$. Denote by $H$ the subgroup of $G$ generated by $A$. 
We can apply to $H$ the structure theorem for finitely generated abelian groups, obtaining that $H$ 
is isomorphic to a subgroup of $\Z^{k}$. So, we can view $A$ as a nice subset of $\Z^k$.
Since, by hypothesis, we are assuming the validity of Alspach's conjecture in $\Z_p$  for infinitely many primes $p$ and 
for any subset of size $k$, by Propositions \ref{prop:Z} and \ref{prop:Zn},
Alspach's conjecture holds in $\Z^k$ for any subset of size $k$. In particular it holds for $A$.
\end{proof}

Now, from Theorem \ref{torsionfree} and the results cited in the introduction we obtain:

\begin{cor}\label{torsionfree10}
Alspach's conjecture holds for any subset of size $k\leq 11$ of any 
torsion-free abelian group. 
\end{cor}

\section{An asymptotic result}\label{asy}

Given an element $g$ of an abelian group $G$, we denote by $o(g)$ the cardinality of the cyclic subgroup $\langle g 
\rangle$ generated by $g$. Furthermore, we set 
$$\vartheta(G)=\min_{0_G\neq g \in G} o(g).$$
Now we are ready to prove that, if $\vartheta(G)$ is large enough, Alspach's conjecture is true for $k\leq 11$. 
This result can be deduced from the compactness theorem of the first order logic but, here, we give a more direct proof.
\begin{thm}\label{thm:as}
Under the hypotheses of {\rm Theorem \ref{torsionfree}}, there exists a positive integer $N(k)$ such that 
Alspach's conjecture holds for any subset of size $k$ of any abelian group $G$ such that $\vartheta(G)>N(k)$.
\end{thm}

\begin{proof}
Let us suppose, for sake of contradiction, that  such $N(k)$ does not exist.
It means that, for any positive integer $M$, there exists an abelian  group $G_M$ such that $\vartheta(G_M)>M$ and 
there exists a nice subset 
$A_M=\{a_{M,1},a_{M,2},\ldots,a_{M,k}\}$ of $G_M$ of size $k$ that contradicts  Alspach's conjecture.
Therefore, for any ordering $\omega=(a_{M,j_1}, a_{M,j_2}, \ldots, a_{M,j_k})$ of 
$A_M$, there exists a pair $(i,j)$, with $i,j \in [1,k]$, such that $s_i(\omega)=s_j(\omega)$. 
Choosing for each $\omega$ one of these pairs, for any positive integer $M$ we can define the function 
$f_M: \Sym(k) \to [1,k]\times [1,k]$ that maps 
$\sigma_\omega=(j_1,j_2,\ldots,j_k)$ into this chosen pair.
Since there are only finitely many  maps from $\Sym(k)$ to $[1,k]\times [1,k]$, there exists an infinite sequence 
$M_1,\ldots, M_n,\ldots$ such that $f_{M_1}=f_{M_i}$ for all $i\geq 1$. 

Let us consider the group 
$G=\bigtimes\limits_{i=1}^{\infty} G_{M_i}$ and the following equivalence relation on $G$.
Given $x=(x_i),y=(y_i)\in G$, we set $x\approx y$ whenever $x_i\neq y_i$  only on a finite number of indices $i$.
Since the equivalence class $[0]$ consists of the elements $(x_i)$ of $G$ that are nonzero on a 
finite number of coordinates $x_i$, and so it is a subgroup of $G$, 
the quotient set  $H=G/\approx$  is still an abelian group. 

Now we want to prove that $H$ is torsion-free. Let us suppose for sake of contradiction that there exists 
an element $[0]\neq [x]\in H$ of finite order, say $n$. Let $\pi_j: G\to G_{M_j}$ be the canonical projection on $G_{M_j}$.
For any $i$ such that $M_i>n$, either $\pi_i(x)=0_{G_{M_i}}$ or 
we have $n\cdot \pi_{i}(x)\not=0_{G_{M_i}}$. However, 
since $n\cdot [x]=[0]$ in $H$ and due to the definition of $\approx$,
we should have $n\cdot \pi_{i}(x)=0_{G_{M_i}}$, for $i$ large enough. It follows that $\pi_{i}(x)$ is eventually zero but 
this is a contradiction since $[x]$ is nonzero. Therefore $H$ is torsion-free.
 
Now we consider the following subset $A$ of $H$:
$$A=\{[z_1],[z_2],\dots,[z_k]\} \mbox{ where } (z_j)_\ell=a_{ M_{\ell},j}.$$
Clearly, $A$ is a nice subset. Given an ordering 
$\omega=([z_{j_1}],[z_{j_2}],\ldots,[z_{j_k}])$ of $A$ we define the ordering 
$\omega_\ell=(a_{M_{\ell},j_1},a_{ M_{\ell},j_2},\ldots,$ $a_{M_\ell, j_k})$ of $A_{M_\ell}$ that 
corresponds to the same permutation $\sigma_\omega$. 
As $f_{M_1}=f_{M_\ell}$ for all $\ell\geq 1$, there exists a pair 
$(i,j)$ such that $s_i(\omega_\ell)=s_j(\omega_\ell)$ for all $\ell$. 
Since $(s_i(\omega_1),s_i(\omega_2),\dots, s_i(\omega_\ell),\dots)$ belongs to the equivalence class $s_i(\omega)$, it easily 
follows that $s_i(\omega)=s_j(\omega)$ also for the set $A$. It means that $A$ is a 
counterexample to Alspach's conjecture. Since this is in contradiction with 
Theorem \ref{torsionfree}, we have proved the statement. 
\end{proof}

\begin{cor}\label{cor:as}
Let $k\leq 11$ be a positive integer. Then, there exists a positive integer $N(k)$ such that  Alspach's conjecture holds 
in $\mathbb{Z}_n$ for any subset of size $k$ whenever the prime factors of $n$ are all greater than $N(k)$.
\end{cor}

\section{Implications on other conjectures}\label{altre}

We consider here two conjectures related to the Alspach's one. These conjectures have been recently studied mainly in 
relation 
to (relative) Heffter arrays and their application for constructing cyclic cycle decompositions of (multipartite) 
complete graphs, see \cite{A, ADDY, CDDY, CMPPHeffter, RelH, DW}.

\subsection{The G-ADMS conjecture}

In \cite{JOS} the validity of Conjecture \ref{Conj:als} was proved for any cyclic group $\Z_p$, where $p$ is a prime, 
whenever $k=|A|\leq 10$. This result was achieved 
using a polynomial method based on the Alon's combinatorial Nullstellensatz. 

\begin{thm}\cite[Theorem 1.2]{Alon}\label{thm:alon}
Let $\F$ be a field and let $f=f(x_1,\ldots,x_k)$ be a polynomial in $\F[x_1,\ldots,x_k]$. Suppose the degree of $f$ is 
$\sum\limits_{i=1}^k t_i$, 
where each $t_i$ is a nonnegative integer, and suppose the coefficient of
$\prod\limits_{i=1}^k x_i^{t_i}$ in $f$ is nonzero.
Then, if $A_1,\ldots,A_k$ are subsets of $\F$ with $|A_i|> t_i$, there are $a_1 \in A_1,\ldots,a_k \in A_k$ so that 
$f(a_1,\ldots,a_k)\neq 0$.
\end{thm}

In order to apply this theorem for proving Alspach's conjecture for subsets of $\Z_p$ of size $k$, J. Hicks,  M.A. 
Ollis 
and J.R. Schmitt constructed a suitable homogeneous polynomial $F_k$ of degree $k(k-1)-1$, identifying a monomial with 
nonzero coefficient such that the degree of each of its terms $x_i$ is less than $|A|=k$ (clearly, we may assume 
$p>k$). 
The existence of values $a_1,\ldots,a_k\in A$ such that
$F_k(a_1,\ldots,a_k)\neq 0$, given by Theorem \ref{thm:alon}, implies that
$\omega=(a_1,\ldots,a_k)$ is an ordering of the elements of $A$ satisfying the requirements of Conjecture 
\ref{Conj:als}.

We recall here the polynomials given in \cite{JOS}. For every $k\geq 2$, let 
$$g_k(x_1,\ldots,x_k)=\prod_{1\leq i < j \leq k} (x_j-x_i)(x_i+\ldots+x_j)\in \Z[x_1,\ldots,x_k]$$
and
\begin{equation}\label{Fk}
F_k(x_1,\ldots,x_k)=\frac{g_{k}(x_1,\ldots,x_k) }{x_1+\ldots+x_k} \in \Z[x_1,\ldots,x_k] .
\end{equation}
Observe that we can also define $F_k$ recursively.
Define $G_\ell\in \Z[x_1,x_2,\ldots,x_{\ell+1}]$, 
where $2\leq \ell\leq k-1$, as follows:
$$G_\ell=(x_1+\ldots+x_\ell)(x_{\ell+1} - x_1)\cdot \prod_{i=2}^\ell 
(x_{\ell+1}-x_i)(x_i+\ldots+x_{\ell+1}).$$
Then,  $F_\ell\in \Z[x_1,x_2,\ldots,x_{\ell}]$ can be defined as
$$F_2=x_2-x_1\quad \textrm{ and  } \quad F_{\ell}=F_{\ell-1}\cdot G_{\ell-1} \quad \textrm{ for } 3\leq \ell\leq k.$$

Now, consider the monomial 
$$\wt c_{k,j}=c_{k,j} x_1^{k-1}\cdots x_{j-1}^{k-1}\cdot x_j^{k-2}\cdot  x_{j+1}^{k-1}\cdots x_k^{k-1}$$
of $F_k$, where $1\leq j \leq k$ and $c_{k,j} \in \Z$.
In \cite[Table 1]{JOS} the authors described the coefficients $c_{k,j}$ for $k\leq 10$, showing 
that either $\gcd(c_{k,1},\ldots,c_{k,k})$ equals to $1$ or its prime factors are all less than $k$. 
This means that, 
for any prime $p>k$, there exists a coefficient $c_{k,j}$ which is nonzero modulo $p$, proving the validity of 
Conjecture \ref{Conj:als} for subsets of $\Z_p$ of size $k$.

We  can use Alon's combinatorial Nullstellensatz  to prove  G-ADMS conjecture.
Also in this case, we use the polynomial provided by \cite{JOS}.
For any $k\geq 2$, let
\begin{equation}\label{fk+1}
f_{k+1}(x_1,\ldots,x_{k+1})=g_{k}(x_2,x_3,\ldots,x_{k+1})\cdot \prod_{j=2}^{k+1} (x_j-x_1)\in 
\Z[x_1,\ldots,x_{k+1}].
\end{equation}
Note that $f_{k+1}$ is an homogeneous polynomial of degree $k^2$.
By Theorem \ref{thm:alon}, to prove the validity of Conjecture \ref{Conj:ADMS} for subsets of $\Z_p$ of size 
$k+1$,
it suffices to find a monomial of $f_{k+1}$ with 
nonzero coefficient such that the degree of each of its terms $x_i$ is less than $k+1$.

Instead of working directly with the polynomial $f_{k+1}$, we show how to use the polynomial $F_k$  also for 
proving the validity of G-ADMS conjecture via Alon's combinatorial Nullstellensatz.
The main advantage is computational: instead of working with a polynomial of degree $k^2$ in $k+1$ indeterminates, 
we work with a polynomial of degree $k^2-k-1$ in $k$ indeterminates.
Hence, consider the monomial
$$\wt d_{k+1,j}=d_{k+1,j}\cdot x_1^{k}\cdots x_{j-1}^{k}\cdot  x_{j+1}^{k}\cdots x_{k+1}^{k}$$
of $f_{k+1}(x_1,\ldots,x_{k+1})$, where $j \in [1, k+1]$ and $d_{k+1,j}\in \Z$.

In order to apply Theorem \ref{thm:alon}, we would like to determine the values of the coefficients $d_{k+1,j}$.
To this purpose, take the monomial
$$\wt e_{k,j}=e_{k,j}\cdot x_1^{k}\cdots x_{j-1}^{k}\cdot  x_{j+1}^{k}\cdots x_{k}^{k}$$
of $g_{k}(x_1,\ldots,x_k)$, where $j \in [1, k]$ and $e_{k,j}\in \Z$.
If $j\geq 2$, then  $x_1^k$ divides $\wt d_{k+1,j}$ and so, by \eqref{fk+1}, the value of $(-1)^k\cdot d_{k+1,j}$ coincides with
the coefficient of $x_2^{k}\cdots x_{j-1}^{k}\cdot  x_{j+1}^{k}\cdots x_{k+1}^{k}$ in $g_k(x_2,\ldots,x_{k+1})$.
In other words, we obtain that 
$$d_{k+1,j+1}=(-1)^k e_{k,j} \quad \textrm{ for all } j \in [1,\ldots,k].$$

Hence, the problem of computing the values of the coefficients $d_{k+1,j+1}$  is equivalent to the problem 
of computing the coefficients $e_{k,j}$  of $g_k=(x_1+\ldots+x_k)\cdot F_k$, see \eqref{Fk}.
So, for any $i,j \in [1,k]$ such that $i\neq j$,  take the monomial
$$\wt a^{(k)}_{i,j}=a^{(k)}_{i,j}\cdot x_i^{k-1} \cdot \prod_{\substack{1\leq r\leq k\\ r\neq i,j}} x_r^k $$
of $F_k$, where $a^{(k)}_{i,j}\in \Z$.
Then, for all $j \in [1,k]$ we have
$$(-1)^k d_{k+1,j+1}=e_{k,j}=a^{(k)}_{1,j}+\ldots + a^{(k)}_{j-1,j}+ a^{(k)}_{j+1,j}+\ldots + a^{(k)}_{k,j}.$$

For instance, using the routines for Magma given in Appendix \ref{app}, we determine the values of $a_{i,j}^{(6)}$, 
see Table 
\ref{tab:6}.
More in general, for $k \in [3,10]$, we obtain that
$$e_{k,j}=(-1)^{\left\lfloor \frac{k-1}{2} \right\rfloor } c_{k,j} \quad \textrm{ for all } j \in [1,k].$$
It would be very interesting to prove this equality for every value of $k$.

\begin{table}[ht]
$$\begin{array}{c|rrrrrr|r} 
j \setminus i & 1 &  2 & 3 & 4 & 5 & 6 & e_{6,j}\\ \hline
1 &  {} &  -28 &   -40 &   -20 &   20  &   40 & -28\\ 
2 &  28  & & -28 & -40 &  -20 &   20 &  -40 \\
3 &  40 &  28 & &   -28 &   -40 &   -20  & -20\\
4 &  20 &    40  &  28 & &   -28 &   -40 & 20 \\
5 &   -20 &  20 &   40 &  28 & &   -28 & 40 \\ 
6 &   -40 &   -20&  20 &   40 &   28  & & 28
\end{array}$$
\caption{Values of the coefficients $a_{i,j}^{(6)}$.}\label{tab:6}
\end{table}

Considering the case $k=11$, with the same routines we obtain the values of the coefficients $e_{11,1}$ and $e_{11,2}$:
$$e_{11,1} = 18128730243333160,\qquad 
  e_{11,2} = 46383022877233608.$$
Since $\gcd(e_{11,1},e_{11,2}  )=2^3$,  the following result follows.

\begin{prop}\label{ADMS12}
{\rm G-ADMS} conjecture holds for subsets of size $k\leq 12$ of cyclic groups of prime order.
\end{prop}
%
%
%
%
%
%
%
%

One can easily adjust to G-ADMS conjecture the arguments 
of Sections \ref{torsion} and \ref{asy}. Clearly, some small modifications are required.
In particular, given an abelian group $(G,+)$, a finite subset $A$ of $G$ is nice for the 
G-ADMS conjecture if $0_G\not \in A$; also, given a  nice subset $A$ of $G$, let
$$\Upsilon(A)=A\cup\Delta(A).$$
Hence, from Proposition \ref{ADMS12} we can deduce  the following results.

\begin{cor}
Given a torsion-free abelian group $G$, {\rm G-ADMS} conjecture is true for any subset $A\subset G\setminus \{0_G\}$ 
such that $|A|\leq 12$.
\end{cor}

\begin{cor}
There exists a positive integer $N$ such that {\rm G-ADMS} conjecture is true for any subset $A$ of 
$\mathbb{Z}_n\setminus\{0\}$ whenever $|A|\leq 12$ and the prime factors of $n$ are all greater than $N$.
\end{cor}

\subsection{The CMPP-conjecture}

In \cite{CMPP}, the following conjecture was proposed by the authors of the present paper, in collaboration with F. Morini and 
A. Pasotti.

\begin{conj}[CMPP]\label{Conj:nostra}
Let $G$ be an  abelian group. Let $A$ be a finite subset of $G\setminus\{0_G\}$
such that no $2$-subset $\{x,-x\}$ is contained in $A$ and
with the property that $\sum_{a\in A} a=0_G$.
Then there exists an ordering of the elements of $A$ such that the partial sums are all distinct.
\end{conj}

Suppose that Conjecture \ref{Conj:als} holds for subsets of size $k$ of a given abelian group $G$.
Let $A$ be a $(k+1)$-subset of $G\setminus\{0_G\}$ such that $\sum\limits_{a\in A} a=0_G$.
Clearly, for any $a \in A$, the set $A\setminus \{a\}$ is a nice subset of $G$ of size $k$:
we can find an ordering $\omega=(a_1,a_2,\ldots,a_k)$
such that all the partial sums $s_i(\omega)$ are nonzero and pairwise distinct ($1\leq i \leq k)$.
Now, taking $\omega'=(a_1,a_2,\ldots,a_k, a)$, we obtain an ordering of the elements of $A$ such that 
$s_i(\omega')\neq s_j(\omega')$ for all $1\leq i<j\leq k+1$.
Therefore, Conjecture \ref{Conj:als} for sets of size at most $k$ in an abelian  group  implies Conjecture 
\ref{Conj:nostra} for sets of size at most $k+1$  in the same group.
It follows that:

\begin{cor}
Given a torsion-free abelian group $G$, 
$\mathrm{CMPP}$ conjecture is true for any subset $A\subset G\setminus\{0_G\}$ such 
that $|A|\leq 12$, $\sum_{z\in A} z=0_G$ and $A$ does not contain pairs of type $\{x,-x\}$.
\end{cor}

\begin{cor}
There exists a positive integer $N$ such that $\mathrm{CMPP}$ conjecture is true for any subset $A$ of 
$\mathbb{Z}_n\setminus\{0\}$, whenever $|A|\leq 12$, $\sum_{z\in A} z=0$, $A$ does not contain pairs of type $\{x,-x\}$ 
and the prime factors of $n$ are all greater than $N$.
\end{cor}

\appendix
\section{Routines for Magma}\label{app}

\begin{verbatim}
Cut:=function(pol,x,k)
 p:=0*x; C:=Coefficients(pol,x);
 for r in [1..Minimum([k+1,#C])] do  p:=p+C[r]*x^(r-1); end for;
return p;
end function;

CutAll:=function(pol,x,i,j)
 B:=Cut(pol,x[j],0);  B:=Cut(B,x[i],#x-1);
 for r in Exclude(Exclude([1..#x],i),j) do  B:=Cut(B,x[r],#x); end for;
return B;
end function;

MultiplyPol:=function(p1,p2,x,i,j,r)
 m1:=Coefficients(p1,r);  m2:=Coefficients(p2,r);
 k:=#x; s:=0*p1;
 for a in [1..Minimum(#m1,k+1)] do
  for b in [1..Minimum(#m2,k+1)] do
   if (a+b-2) le k then 
    s:=s+CutAll(m1[a]*x[r]^(a-1)*m2[b]*x[r]^(b-1),x,i,j);
   end if;
  end for; 
 end for;
return s;
end function;

PolyGk:=function(i,j,x,k)
 p:=1;
 for a in [2..k] do
  s:=0; for b in [a..k+1] do s:=s+x[b]; end for;
  p:=CutAll(p*(x[k+1]-x[a])*s,x,i,j);
 end for;
 s:=0; for b in [1..k] do s:=s+x[b]; end for;
 B:=CutAll(p*(x[k+1]-x[1])*s,x,i,j);
return B;
end function;

R<x1,x2,x3,x4,x5,x6,x7,x8,x9,x10,x11>:=PolynomialRing(Integers(),11);
x:=[x1,x2,x3,x4,x5,x6,x7,x8,x9,x10,x11];   F2:=R!(x2-x1);; 

for j in [1,2] do
 s:=0;
 for i in Exclude([1..#x],j) do
  A:=F2;
  for r in [2..#x-1] do 
   A:=MultiplyPol(A,PolyGk(i,j,x,r),x,i,j,r);  
  end for;
  s:=s+x[i]*A;
 end for;
 "Monomial e_{",#x,",",j,"}: ",s;
end for;

\end{verbatim}

With the routine \texttt{CutAll} we ignore the monomials 
$c\, x_1^{m_1}x_2^{m_2}\cdots x_k^{m_k}$ such that $m_j>0$, $m_i>k-1$ and  $m_r >k$, where  $r\neq i,j$.


\begin{thebibliography}{30}

\bibitem{Alon} N. Alon, 
Combinatorial Nullstellensatz,
\textit{Combin. Probab. Comput.} \textbf{8} (1999), 7--29. 

\bibitem{Al} B. Alspach, D.L. Kreher, A. Pastine, 
The Friedlander-Gordon-Miller conjecture is true, 
\textit{Australas. J. Combin.} \textbf{67} (2017), 11--24.

\bibitem{6} B. Alspach, G. Liversidge, 
On strongly sequenceable abelian groups, 
\textit{Art Discrete Appl. Math.}, to appear.

\bibitem{A} D.S. Archdeacon,
Heffter arrays and biembedding graphs on surfaces,
\textit{Electron. J. Combin.} \textbf{22} (2015), \#P1.74.

\bibitem{ADDY} D.S. Archdeacon, J.H. Dinitz, D.M. Donovan, E.S. Yaz\i c\i,
Square integer Heffter arrays with empty cells,
\textit{Des. Codes Cryptogr.} \textbf{77} (2015), 409--426.

\bibitem{ADMS} D.S. Archdeacon, J.H. Dinitz, A. Mattern, D.R. Stinson, 
On partial sums in cyclic groups, 
\textit{J. Combin. Math. Combin. Comput.} \textbf{98} (2016),  327--342. 

\bibitem{BH} J.-P. Bode, H. Harborth,
Directed paths of diagonals within polytopes,
\textit{Discrete Math.} \textbf{299} (2005), 3--10.
 
\bibitem{mg} W. Bosma, J. Cannon, C. Playoust, 
The Magma algebra system. I. The user language. 
\textit{J. Symbolic Comput.} \textbf{24} (1997), 235--265. 


\bibitem{CDDY} N.J. Cavenagh, J.H. Dinitz, D.M. Donovan, E.\c{S}. Yaz\i c\i,
The existence of square non-integer Heffter arrays,
\textit{Ars Math. Contemp.} \textbf{17} (2019), 369--395. 

\bibitem{CMPP}  S. Costa, F. Morini, A. Pasotti  M.A. Pellegrini,
A problem on partial sums in abelian groups,
\textit{Discrete Math.} \textbf{341} (2018), 705--712.

\bibitem{CMPPHeffter} S. Costa, F. Morini, A. Pasotti, M.A. Pellegrini,
Globally simple Heffter arrays and orthogonal cyclic cycle decompositions,
\textit{Austral. J. Combin.} \textbf{72} (2018), 549--493.

\bibitem{RelH} S.Costa, F. Morini, A. Pasotti, M.A. Pellegrini,
A generalization of Heffter arrays,
\textit{J. Combin. Des.} \textbf{28} (2020), 171--206.

\bibitem{DW} J.H. Dinitz, I.M. Wanless,
The existence of square integer Heffter arrays,
\textit{Ars Math. Contemp.} \textbf{13} (2017), 81--93. 

\bibitem{Gordon} B. Gordon, 
Sequences in groups with distinct partial products,  
\textit{Pacific J. Math.} \textbf{11} (1961), 1309--1313.
 
\bibitem{Gr} R.L. Graham, 
On sums of integers taken from a fixed sequence,
In: J.H. Jordan, W. A. Webb (eds), 
\textit{Proceedings of the Washington State University Conference on Number Theory}
(Washington State Univ., Pullman, Wash., 1971), pp. 22--40. 
Dept. Math.; Pi Mu Epsilon, Washington State University, Pullman, Wash., 1971.

\bibitem{JOS} J. Hicks,  M.A. Ollis, J.R. Schmitt,
Distinct partial sums in cyclic groups: polynomial method and constructive approaches,
\textit{J. Combin. Des.} \textbf{27} (2019), 369--385.

\bibitem{O}  M.A. Ollis, 
Sequenceable groups and related topics, 
\textit{Electron. J. Combin.} \textbf{20}  (2013), \#DS10v2. 

\bibitem{O1} M.A. Ollis,
Sequences in dihedral groups with distinct partial products, preprint
available at https://arxiv.org/abs/1904.07646.

\bibitem{Sarah} M.A. Ollis, S. Rovner-Frydman, J.R. Schmitt, private communication (2020).

\end{thebibliography}
\end{document}